\documentclass[11pt,letterpaper]{article}

\usepackage{amsmath,amssymb,amsfonts,amsthm,mathrsfs}
\usepackage{graphicx}
\usepackage[usenames,dvipsnames]{color}

\usepackage[margin=1in]{geometry}

\theoremstyle{plain}
\newtheorem{thm}{Theorem}
\newtheorem{lem}[thm]{Lemma}
\newtheorem{cor}[thm]{Corollary}
\newtheorem{prop}[thm]{Proposition}
\newtheorem{remark}[thm]{Remark}

\theoremstyle{definition}
\newtheorem{obs}[thm]{Observation}
\theoremstyle{remark}

\newcommand{\real}{\ensuremath {\mathbb R} }

\newcommand{\nat}{\ensuremath {\mathbb N} }

\newcommand{\remove}[1] {}

\newcommand{\ex} {{\bf E}}
\newcommand{\pr} {{\bf Pr}}
\newcommand{\var} {{\bf Var}}

\newcommand{\cD} {\ensuremath{\mathcal D}}
\newcommand{\sG} {\ensuremath{\mathscr G}}
\newcommand{\G}{{\mathcal{G}}}
\newcommand{\eps}{\varepsilon}

\newcommand{\p}{{\widehat p}}
\DeclareMathOperator{\Bin}{Bin}

\title{The bondage number of random graphs}
\author{Dieter Mitsche\thanks{e-mail: \texttt{dmitsche@unice.fr}, Universit\'{e} de Nice Sophia-Antipolis, Laboratoire J.A.~Dieudonn\'{e}, Parc Valrose, 06108 Nice cedex 02, France} 
\and Xavier P\'erez-Gim\'enez\thanks{e-mail: \texttt{xperez@ryerson.ca}, Department of Mathematics, Ryerson University, Toronto, ON, Canada} 
\and Pawe{\l} Pra{\l}at\thanks{e-mail: \texttt{pralat@ryerson.ca}, Department of Mathematics, Ryerson University, Toronto, ON, Canada; research partially supported by NSERC and Ryerson University}}
\date{}

\begin{document}

\maketitle

\begin{abstract}
A dominating set of a graph is a subset $D$ of its vertices such that every vertex not in $D$ is adjacent to at least one member of $D$. The domination number of a graph $G$ is the number of vertices in a smallest dominating set of $G$. The bondage number of a nonempty graph $G$ is the size of a smallest set of edges whose removal from $G$ results in a graph with domination number greater than the domination number of $G$. In this note, we study the bondage number of the binomial random graph $\sG(n,p)$. We obtain a lower bound that matches the order of the trivial upper bound. As a side product, we give a one-point concentration result for the domination number of $\sG(n,p)$ under certain restrictions.
\end{abstract}

\section{Introduction}

In this paper, we consider the \textbf{Erd\H{o}s-R\'enyi random graph process}, which is a stochastic process that starts with $n$ vertices and no edges, and at each step adds one new edge chosen uniformly at random from the set of missing edges. Formally, let $e_1, e_2, \ldots, e_{n \choose 2}$ be a random permutation of the edges of the complete graph $K_n$. The graph process consists of the sequence of random graphs $( \G(n,m) )_{m=0}^{n \choose 2}$, where $\G(n,m)=(V,E_m)$, $V = [n] := \{1, 2, \ldots, n\}$, and $E_m = \{ e_1, e_2, \ldots, e_m \}$. It is clear that $\G(n,m)$ is a graph taken uniformly at random from the set of all graphs on $n$ vertices and $m$ edges (see, for example,~\cite{bol, JLR} for more details.)

Our results refer to the random graph process. However, it will be sometimes easier to work with the $\sG(n,p)$ model instead of $\G(n,m)$. The (binomial) \textbf{random graph} $\sG(n,p)$ consists of the probability space $(\Omega, \mathcal{F}, \Pr)$, where $\Omega$ is the set of all graphs with vertex set $[n]$, $\mathcal{F}$ is the family of all subsets of $\Omega$, and for every $G \in \Omega$,
$$
\Pr (G) = p^{|E(G)|} (1-p)^{{n \choose 2} - |E(G)|} \,.
$$
This space may be viewed as the set of outcomes of ${n \choose 2}$ independent coin flips, one for each pair $\{u,v\}$ of vertices, where the probability of success (that is, adding edge $uv$) is $p.$ Note that $p=p_n$ may (and usually does) tend to zero as $n$ tends to infinity.  

\medskip

All asymptotics throughout are as $n \rightarrow \infty$ (we emphasize that the notations $o(\cdot)$ and $O(\cdot)$ refer to functions of $n$, not necessarily positive unless otherwise stated, whose growth is bounded; on the other hand, functions hidden in $\Theta(\cdot)$ and $\Omega(\cdot)$ notations are positive). We use the notation $a_n\sim b_n$ to denote $a_n=(1+o(1))b_n$. A sequence $a_n$ satisfies a certain property \textbf{eventually} if the property holds for all but finitely many terms of the sequence. We say that an event in a probability space holds \textbf{asymptotically almost surely} (or \textbf{a.a.s.}) if the probability that it holds tends to $1$ as $n$ goes to infinity. We often write $\G(n,m)$ and $\sG(n,p)$ when we mean a graph drawn from the distribution $\G(n,m)$ and $\sG(n,p)$, respectively. All logarithms in this paper are natural logarithms.

\bigskip

A \textbf{dominating set} for a graph $G = (V, E)$ is a subset $D$ of $V$ such that every vertex not in $D$ is adjacent to at least one member of $D$. The \textbf{domination number}, $\gamma(G)$, is the number of vertices in a smallest dominating set for $G$. The \textbf{bondage number}, $b(G)$, of a non-empty graph $G$ is the smallest number of edges that need to be removed in order to increase the domination number; that is,
$$
b(G) = \min\{ |B| : B \subseteq E, \gamma(G-B) > \gamma(G) \}.
$$
(If $G$ has no edges, then we define $b(G) = \infty$.) 
This graph parameter was formally introduced in 1990 by Fink et al.~\cite{Fink} as a parameter for measuring the vulnerability of the interconnection network under link failure. However, it was considered already in 1983 by Bauer at al.~\cite{Bauer} as ``domination line-stability''. Moreover, graphs for which the domination number changes upon the removal of a single edge were investigated by Walikar and Acharya~\cite{Walikar} in 1979. One of the very first observations~\cite{Bauer, Fink} is the following upper bound:
$$
b(G) \le \min_{xy \in E} \{ \deg(x) + \deg(y) - 1 \} \le \Delta(G) + \delta(G)-1,
$$
where $\Delta(G)$ and $\delta(G)$ are the maximum and, respectively, the minimum degree of $G$. Since a.a.s.\ $\Delta(\sG(n,p)) \sim \delta(\sG(n,p)) \sim pn$ provided $pn \gg  \log n$ (this follows immediately from Chernoff's bound stated below, and the union bound), we get that a.a.s.\ 
\begin{equation}\label{eq:trivialupper1}
b(\sG(n,p)) \le 2 pn (1+o(1))
\end{equation}
for $pn \gg \log n$. For denser graphs, one can improve the leading constant of this upper bound by using the following observation of Hartnell and Rall~\cite{Hartnell}:
$$
b(G) \le \min_{xy \in E} \{ \deg(x) + \deg(y) - 1 - |N(x) \cap N(y)| \}.
$$
It follows that if $p = \Omega(1)$, then a.a.s.\
$$
b(\sG(n,p)) \le (2 p - p^2) n (1+o(1)).
$$
Today, many properties of the bondage number are studied. For more details the reader is directed to the survey~\cite{survey} which cites almost 150 papers on the topic.  

\section{Results}

Our goal is to investigate the bondage number of the binomial random graph on $n$ vertices and of the random graph process. Throughout the whole paper we will exclude the case $p = p_n \to 1$ and also assume that $p$ does not tend to zero too fast. More precisely, our main results require that $p=p_n$ eventually satisfies
\begin{equation}\label{eq:pdef}
n^{-1/3+\eps} \leq p \leq 1-\varepsilon,
\end{equation}
for some constant $\varepsilon > 0$, but most arguments only require the following, milder, constraint:
\begin{equation*}
\log^2 n/\sqrt{n} \ll  p \leq 1-\varepsilon.
\end{equation*}
Since our results are asymptotic in $n$, we will assume that $n$ is large enough so that all requirements in the argument are met. (In particular, the notation ``eventually'' is often implicitly assumed in the proofs and omitted.) Let $\cD_k$ be the set of dominating sets of size $k$ of $\sG(n,p)$, and let $X_k=|\cD_k|$. Clearly,
\begin{equation}\label{eq:fdef}
f(n,k,p) := \ex X_k = \binom{n}{k} \left(1-(1-p)^k\right)^{n-k}.
\end{equation}
For a given $p = p_n$, let
\begin{equation}\label{eq:rdef}
r = r_n = \min\{k \in\nat: f(n,k,p) > 1/(pn) \}.
\end{equation}
Since $pn \gg \sqrt{n}\log^2 n > 1$ (eventually) and $f(n,n,p) = 1$, the function $r$ is well defined for $n$ sufficiently large.

\subsection{Random Graph Process}

Consider the random graph process $( \G(n,m) )_{0\le m \le {n \choose 2} }$. Clearly, the random variable $\gamma( \G(n,m) )$ is a non-increasing function of $m$, $\gamma( \G(n,0) ) = \gamma( \overline{K}_n ) = n$, and $\gamma( \G(n,{n \choose 2}) ) = \gamma(K_n) = 1$. Suppose that at some point the domination number drops down, that is, there exists a value of $m$ such that $\gamma( \G(n,m) ) = k+1$ but $\gamma( \G(n,m+1) ) = k$. The random graph process continues and, as long as the domination number remains to be equal to $k$, the bondage number, $b( \G(n,m+\ell) )$, is a non-decreasing function of $\ell$. Moreover, we get that $b( \G(n,m+\ell) ) \le \ell$, as one can remove the last $\ell$ edges that were added in the process (namely, $e_{m+1}, e_{m+2}, \ldots, e_{m+\ell}$) in order to increase the domination number. A natural and interesting question is then to ask how large the bondage number is right before the domination number drops again; that is, what can be said about $b( \G(n,m+\ell) )$ when $\gamma( \G(n,m+\ell) ) = k$ but $\gamma( \G(n,m+\ell+1) ) = k-1$? It turns out that, for the range of $k$ we are interested in, it is of the order of the maximum degree of $\G(n,m+\ell)$, and hence it matches the trivial, deterministic, upper bound mentioned in the introduction (up to a constant multiplicative factor). Here is the precise statement.
\begin{thm}\label{thm:main_gnd}
Given any constant $\eps > 0$, let $k=k_n$ be such that eventually $ \eps \log n \le k \le n^{1/3-\eps}$. Then, there exists $m=m_n$ such that a.a.s.\ 
$$
\gamma( \G(n,m) ) = k \hspace{0.5in} \text{and} \hspace{0.5in} b( \G(n,m) ) = \Theta(\Delta( \G(n,m) )) = \Theta(m/n).
$$
\end{thm}

\subsection{Binomial Random Graph}

Consider now the binomial random graph $\sG(n,p)$. Before we state the main result for this probability space, let us mention some technical difficulties one needs to deal with. Our one-point concentration result (below) on the domination number of $\sG(n,p)$ amounts to showing that a.a.s.\ $X_r\ge1$ (since, trivially, a.a.s.\ $X_i=0$ for all $i\le r-1$). Moreover, our claim about the bondage number requires that a.a.s.\ $X_r = \Omega(pn) \to \infty$. This follows from the fact that the number of dominating sets of minimum cardinality is an upper bound on the bondage number (since each such set $D$ must have a vertex $v\notin D$ adjacent to only one vertex in $D$, and thus $D$ can be neutralized by removing a single edge). Therefore, we will restrict ourselves to situations in which $\ex X_r$ is large enough and prove concentration of $X_r$ around its mean. For technical reasons of the argument, we will require the aforementioned condition to hold for two consecutive values $n-1$ and $n$ (see~\eqref{eq:condI} and~\eqref{eq:Iprime} in Theorem~\ref{thm:main_general}). This motivates the assumptions on the ratio $p_{n+1}/p_n$ in the statement of Theorem~\ref{thm:main_gnp}.

We point out that, even for ``natural'' functions satisfying our assumptions, such as $p_n=1/2$, it is not clear whether there are always many dominating sets of minimum cardinality, or rather $X_{r_n}$ oscillates reaching both small and large values as $n$ grows. All we managed to show is that, for such $p_n$, almost all values of $n$ satisfy~\eqref{eq:condI} and~\eqref{eq:Iprime} and thus yield the bondage number as large as possible. To make this precise, a set $I\subseteq\nat$ is said to be {\bf dense} if
\begin{equation}\label{eq:densedef}
\lim_{n\to\infty} \dfrac{|I\cap[n]|}{n}=1.
\end{equation}
In view of this definition, and recalling the definition of $r=r_n$ in~\eqref{eq:rdef}, our result for the binomial random graph can be stated as follows.
\begin{thm}\label{thm:main_gnp}
Given any constant $\eps > 0$, let $p=p_n$ be such that eventually $n^{-1/3+\eps} \le p \le 1-\eps$. Moreover, suppose there exists a non-increasing non-negative sequence $h=h_n$ such that $p_{n+1}/p_n=1-\Theta(h_n/n)$. Then, there exists a dense set $I\subseteq\nat$ such that, with asymptotics restricted to $n\in I$, a.a.s.\
$$
\gamma(\sG(n,p)) = r
\qquad\text{and}\qquad
b( \sG(n,p) ) = \Theta(\Delta( \G(n,p) )) = \Theta(pn).
$$
\end{thm}

Although the conditions on $p_n$ in Theorem~\ref{thm:main_gnp} seem restrictive, many common and natural probability functions $p_n$ satisfy it. For example, $p_n=n^{-1/4}$, $p_n=1/\log \log n$ and $p_n=1/2$ meet the requirements (by picking $h_n=1$, $h_n=1/\log n$ and $h_n=0$, respectively). Other, seemingly more complicated, choices such as $p_n = (n+1)^{-1/4}\log^3 n + n^{-1/3}$ also satisfy our conditions. On the other hand, mixed behaviours such as
\[
p_n=\begin{cases}n^{-1/4}&\text{$n$ even}\\1/\log \log n&\text{$n$ odd}\end{cases}
\]
are not considered here. One can easily relax the conditions on $p_n$ a bit further, but we do not aim for it, as it does not appear to be possible to express that in terms of any ``natural" assumptions such as ``$p_n$ being non-decreasing".

\subsection{General Result}

In fact, both Theorem~\ref{thm:main_gnd} and Theorem~\ref{thm:main_gnp} are implied by the following, slightly more general, result. It is known that even for sparser graphs (namely, for $p = p_n \gg \log^2 n/\sqrt{n}$, but bounded away from~1) a.a.s.\ the domination number of $\sG(n,p)$ takes one out of two consecutive integer values, $r$ or $r+1$, where $r=r_n$ is defined in~\eqref{eq:rdef} (see~\cite{GLS} and also~\cite{WG} for an earlier paper where denser graphs were considered). The next result shows that if $f(n,r,p)$ (that is, the expected number of dominating sets of cardinality $r$) is large, then we actually have one-point concentration and the bondage number is of order $pn$. Note that we may have to restrict asymptotics to an infinite subset of $\nat$ that guarantees our assumptions on $f$.

\begin{thm}\label{thm:main_general}
Given any constant $\eps>0$, suppose that $p=p_n$ eventually satisfies $n^{-1/3+\eps} \leq p \leq 1-\eps$, and let $f$ and $r$ be defined as in~\eqref{eq:fdef} and~\eqref{eq:rdef}. Suppose that there exists an infinite set $I'\subseteq\nat$ and $\omega =\omega_n \to\infty$ such that 
\begin{equation}\label{eq:condI}
\ex X_r = f(n,r,p) \geq \exp\left(\omega \log n\right)
\qquad
\text{(for $n\in I'$).
}
\end{equation}
Then, a.a.s.\ 
\[
\gamma( \sG(n,p) ) = r
\qquad
\text{(for $n\in I'$).
}
\]
Moreover, suppose that
\begin{equation}\label{eq:Iprime}
I = \{n\in\nat : n\in I',\; n-1\in I'\}
\qquad
\text{has infinite cardinality.}
\end{equation}
Then a.a.s.\
\[
b( \sG(n,p) ) = \Theta(\Delta( \sG(n,p) )) = \Theta(pn)\qquad
\text{(for $n\in I$).
}
\]
\end{thm}
\begin{remark}\hspace{0cm}
\begin{itemize}
\item[(i)]
In many applications of Theorem~\ref{thm:main_general} (for instance, in the proofs of Theorems~\ref{thm:main_gnd} and~\ref{thm:main_gnp}), $I'$ is a dense subset of $\nat$. Then, automatically $I$ is also dense, and thus has infinite cardinality as required. 
\item[(ii)]
The first part of the theorem, which characterizes the domination number of $\sG(n,p)$, holds in fact for any $p=p_n$ satisfying $\log^2 n/\sqrt{n} \ll p\le 1-\eps$ (see Corollary~\ref{cor:small_variance} below). 
\end{itemize}
\end{remark}

\bigskip

The paper is structured as follows. In Section~\ref{sec:prelim}, we show that the results for $\G(n,m)$ and $\sG(n,p)$ can be obtained from Theorem~\ref{thm:main_general}. Section~\ref{sec:second} develops some tools required to estimate the second moment of $X_r$ and some other random variables. Finally, Section~\ref{sec:theorem} is devoted to prove Theorem~\ref{thm:main_general}. 

\section{Preliminaries}\label{sec:prelim}

In this section we are going to introduce a few inequalities used in the paper, and we show some properties of the functions $r=r_n$ and $f(n,r,p)$ defined in~(\ref{eq:fdef}) and~(\ref{eq:rdef}). The function $p$ will be assumed to satisfy~(\ref{eq:pdef}). We will also show that Theorem~\ref{thm:main_gnd} and Theorem~\ref{thm:main_gnp} are implied by Theorem~\ref{thm:main_general}.

\bigskip

We will use the following version of Chernoff bound (see~e.g.~\cite{JLR}):
\begin{lem}[\textbf{Chernoff Bound}] 
If $W$ is a binomial random variable with expectation $\mu$, and $0<\delta<1$, then, setting $\varphi(x)=(1+x)\log(1+x)-x$ for $x \geq -1$ (and $\varphi(x)=\infty$ for $x < -1$),
\begin{equation}\label{eq:chernoffstrong}
\pr [W < (1-\delta)\mu]  \le  \exp \left(-\mu \varphi(-\delta) \right) \le \exp \left( -\frac{\delta^2 \mu}{2} \right);
\end{equation}
and if $\delta > 0$,
\begin{equation}\label{eq:chernoffnormal}
\pr [ W > (1+\delta)\mu] \le \exp\left(-\frac{\delta^2 \mu}{2+\delta}\right).
\end{equation}
\end{lem}

\bigskip

Given $p = p_n \in[0,1)$, define $\p=\log\frac{1}{1-p}$. Note that $\p \ge p$ (with equality only holding at $p=0$), and
\begin{equation}\label{eq:hatp}
 \begin{cases}
\p \sim p
& \text{if $p=o(1)$,}
\\
\p = \Theta(1)
& \text{if $p=\Theta(1)$ and $1-p=\Theta(1)$,}
\\
\p \to \infty
& \text{if $p\to1$.}
\end{cases}
\end{equation}

We start with a few simple observations. Let us mention that some of the properties we show below are known and can be found in, for example,~\cite[Observation 2.1]{GLS} (but mainly for $p=o(1)$). We present the proof here for completeness and to prepare the reader for similar calculations later on.

\begin{lem}\label{lem:relations}
Assume $\log^2 n/\sqrt n \ll p \leq 1-\varepsilon$ for some constant $\varepsilon > 0$, and let $r$ be defined as in~\eqref{eq:rdef}. Then, the following holds:
\begin{enumerate}
\item[(i)]
\[
r = \left\lceil \frac1\p\log\left(\frac{\p n}{\log^2(pn)}(1+o(1))\right) \right\rceil =
\begin{cases}
\displaystyle \frac1\p\log\left(\frac{\p n}{\log^2(pn)}(1+o(1))\right)
& \text{if $p=o(1)$}
\\[1.2em]
\displaystyle \frac1\p\log\left(\alpha \frac{\p n}{\log^2(pn)} \right)
& \text{if $p=\Omega(1)$,}
\end{cases}
\]
for some $1+o(1) \le \alpha \le \frac{1+o(1)}{1-p}=\Theta(1)$.

\item[(ii)] 
\[
r = \Theta\left(\frac{\log n}{p}\right)
\qquad\text{and}\qquad
(1-p)^r = \Theta\left(\frac{\log^2 n}{pn}\right).
\]
In particular, $r=\Omega(\log n)$ and $r=o(\sqrt n/\log n)$.
\item[(iii)] Moreover, if $k = r + O(1)$, then
$$
\frac{f(n,k+1,p)}{f(n,k,p)} = \exp(\Theta(\log^2 n)).
$$
\end{enumerate}
\end{lem}
\begin{proof}
For a given function $g=g_n =o(1)$, we define 
$$
s_g = s_g(n,p) = \left\lceil \frac1\p\log\left(\frac{\p n}{\log^2(pn) (1+g_n)}\right) \right\rceil. 
$$
First, observe that for $p$ in the range of discussion, $\log (pn) = \Theta(\log n)$. Then, it follows from~(\ref{eq:hatp}) that
\begin{equation}\label{eq:s_g}
s_g = \frac {1}{\Theta(p)} \Big( \Theta(\log (pn)) - O(\log \log (pn)) \Big) = \Theta \left( \frac {\log n}{p} \right).
\end{equation}
Also, from the definition of $\p$ and~\eqref{eq:hatp}, we obtain $(1-p)^{s_g}= \Theta\left(\frac{\log^2 n}{pn}\right).$ Hence, part~(ii) will follow, once we show that $r=s_g$ for some function $g_n =o(1)$. In particular, proving part~(i) will automatically yield part~(ii).

Given any $g_n=o(1)$, define $g^-_n$ and $g^+_n$ by
\[
s_g =  \frac1\p\log\left(\frac{\p n}{\log^2(pn) (1+g^-_n)}\right) 
\qquad\text{and}\qquad
s_g - 1 =  \frac1\p\log\left(\frac{\p n}{\log^2(pn) (1+g^+_n)}\right).
\]
Since $z\le \lceil z\rceil < z+1$ for any $z\in\real$, we obtain that $g^-_n \le g_n \le g^+_n.$ 
%
%

Now we proceed to estimate $f(n,s_g,p)$ and $f(n,s_g-1,p)$ for any $g_n=o(1)$.
Denoting by $[n]_k=n(n-1)\ldots(n-k+1)$, and using Stirling's formula ($k! \sim \sqrt{2\pi k}(k/e)^k$), we observe that
\begin{eqnarray*}
f(n,s_g,p) &=& \frac {[n]_{s_g}}{s_g!} \left( 1 - \frac {\log^2 (pn)}{\p n} (1+ g^-_n ) \right)^{n(1-s_g/n)} \\
&=& \frac {n^{s_g} (1+O(s_g/n))^{s_g} }{(1+o(1)) \sqrt{2\pi s_g} (s_g/e)^{s_g}} \exp \left( - \frac {\log^2 (pn)}{\p} \left( 1+ g^-_n + O \left(\frac {\log^2 n}{pn} \right) + O \left( \frac {s_g}{n} \right) \right) \right) \\
&=& \frac {1+o(1)}{\sqrt{2\pi s_g}} \exp \left( s_g \log \left( \frac {ne}{s_g} \right) - \frac {\log^2 (pn)}{\p} \left( 1+ g^-_n + O \left(\frac {\log^2 n}{pn} \right) \right) \right),
\end{eqnarray*}
since $s_g = \Theta( \log n / p )$ (by~\eqref{eq:s_g}) and $p \gg \log^2 n / \sqrt{n}$, which implies that 
$$
(1+O(s_g/n))^{s_g} = e^{O(s_g^2/n)} = e^{O(\log^2 n / (p^2 n))} \sim 1
$$ 
and $s_g/n = O(\log n / (pn)) = o(\log^2 n / (pn))$. Hence,
\begin{eqnarray}
f(n,s_g,p) &\sim& \frac {1}{\sqrt{2\pi s_g}} \exp \Bigg( \frac {\log (pn) + O(\log \log n)}{\p} \left( \log \left( \frac {pn}{\log n} \right) + O(1) \right)
\notag\\
&& \hspace{9em}{ } - \frac {\log^2 (pn)}{\p} \left( 1+ g^-_n + O \left(\frac {\log^2 n}{pn} \right) \right) \Bigg)
\notag\\
&=& \Theta \left( \sqrt{ \frac{p}{\log n} } \right) \exp \Bigg( - \frac {\log^2 (pn)}{\p} \left( g^-_n + O \left(\frac {\log \log n}{\log n} \right) \right) \Bigg).
\label{eq:fnsgp}
\end{eqnarray}
Moreover, the same calculations leading to~\eqref{eq:fnsgp} are valid if we replace $s_g$ by $s_g-1$ and $g^-_n$ by $g^+_n$, so we also get
\begin{equation}
f(n,s_g-1,p) = \Theta \left( \sqrt{ \frac{p}{\log n} } \right) \exp \Bigg( - \frac {\log^2 (pn)}{\p} \left( g^+_n + O \left(\frac {\log \log n}{\log n} \right) \right) \Bigg).
\label{eq:fnsgp2}
\end{equation}
In order to prove part~(ii), we first take $g_n = (\log \log n)^2 / \log n$. From~\eqref{eq:fnsgp2} and since $g^+_n \ge g_n$, we obtain

\[
f(n, s_g-1 ,p) \le \exp \left( - \Omega( (\log^2 \log n) \log(pn) ) \right) = o(1/(pn)),
\]
and thus
\[
f(n,j,p) < 1/(pn)
\qquad
\text{for all}\quad
1\le j\le s_g-1,
\]
since $f(n,j,p)$ is increasing with respect to $j$ in that range (this can be easily checked by looking at the ratio $f(n,j+1,p)/f(n,j,p)$ for $j = O(\log n/p) = o(\sqrt n)$).
Therefore, $r > s_g - 1$ and, since both $r$ and $s_g$ are natural numbers, $r \ge s_g =  s_{(\log \log n)^2 / \log n}$.
On the other hand, if we set $g_n = -(\log \log n)^2 / \log n$ then, by~\eqref{eq:fnsgp} and since $g^-_n \le g_n$,
\[
f(n,s_g,p) \gg  n^{-1/4} \sqrt{\log n} \cdot \exp \left( \Omega( (\log^2 \log n) \log(p n)) \right) > 1/(pn),
\]
and hence $r \le s_{- (\log \log n)^2 / \log n}$. Combining the two bounds, we conclude that
\[
r = s_g
\qquad\text{for some}\quad
-(\log \log n)^2 / \log n \le g_n \le (\log \log n)^2 / \log n,
\]
which implies the first equality in part~(i). The second equality follows immediately from setting $\alpha=1/(1+g^-_n)$ and the fact that $\frac{1}{1+g_n} \le \frac{1}{1+g^-_n} < \frac{1}{(1+g_n)(1-p)}$.

Finally, let us move to part (iii). Using part (ii), it is easy to see that for $k = r +O(1)$ we get
\begin{eqnarray*}
\frac{f(n,k+1,p)}{f(n,k,p)} & = & \frac {[n]_{k+1} / (k+1)!}{[n]_k / k!} \left( \frac { 1-(1-p)^{k+1} } { 1-(1-p)^{k} } \right)^{n-k} \left(1-(1-p)^{k+1}\right)^{-1} \\
&\sim& \frac {n}{k} \left( \frac { 1-(1-p)^{k} + p(1-p)^{k}  } { 1-(1-p)^{k} } \right)^{n-k}  \\
&=& \Theta \left( \frac{pn}{\log n} \right) \left( 1 + \Theta(p(1-p)^{k}) \right)^{n-k}  \\
&=& \Theta \left( \frac{pn}{\log n} \right) \exp \left( \Theta(pn(1-p)^{k}) \right) \\
&=& \Theta \left( \frac{pn}{\log n} \right) \exp \left( \Theta( \log^2 n ) \right) = \exp \left( \Theta( \log^2 n ) \right).
\end{eqnarray*}
This finishes the proof of the lemma.
\end{proof}

Now, we will show that Theorem~\ref{thm:main_gnd} can be obtained from Theorem~\ref{thm:main_general}.
\begin{proof}[Proof of Theorem~\ref{thm:main_gnd}]
Let $k=k_n$ be such that $\eps \log n \le k \le n^{1/3-\eps}$ for some $\eps > 0$. Our goal is to show that there exists $m = m_n \in \nat$ such that a.a.s.\ $\gamma( \G(n,m) ) = k$ and $b( \G(n,m) ) = \Theta(\Delta( \G(n,m) )) = \Theta(m/n)$. We assume that Theorem~\ref{thm:main_general} holds and we will use the probability space $\sG(n,p)$ to get the result.

It follows immediately from definition~\eqref{eq:fdef} that, for $1\le j < n$, $f(n,j,p)$ is both a continuous and increasing function of $p$, taking all values between $0$ and $\binom{n}{j}$. Then, given $n\in\nat$ (sufficiently large), we can define $p_+$ to be such that
\begin{equation}\label{eq:ppdef}
f(n,k-1,p_+) = 1/(p_+ n).
\end{equation}
Moreover, straightforward computations show that, for $0<p<1$ and $0\le j\le n/4$,
\begin{equation}\label{eq:easyratio}
\frac{f(n,j,p)}{f(n,j+1,p)} \le \frac{\binom{n}{j}}{\binom{n}{j+1}} = \frac{j+1}{n-j} < 1/2,
\end{equation}
so in particular $f(n,j,p)$ is increasing in $j$, for $j$ in that range. Let $r_+$ be defined as $r$ in~\eqref{eq:rdef} for $p=p_+$. From~\eqref{eq:ppdef} and~\eqref{eq:easyratio}, we deduce that $f(n,k,p_+) > 1/(p_+n)$ and $f(n,j,p_+) \le 1/(p_+n)$ for all $j\le k-1$, so we must have $r_+ = k$.
Also observe that $r$ in~\eqref{eq:rdef} is a non-increasing function of $p$. Combining this fact and Lemma~\ref{lem:relations}(i), we conclude that $n^{-1/3+\eps'} \le p_+\le 1-\eps'$, for some constant $\eps'=\eps'(\eps)$, since otherwise $r_+<\eps\log n$ or $r_+>n^{1/3-\eps}$ contradicting our assumptions on $k$ and the fact that $k=r_+$. Hence, in particular, $1/ (p_+ n)=o(1)$. It follows immediately from the first moment method that a.a.s.\ $\sG(n,p_+)$ has no dominating set of size $k-1$,
and then
\begin{equation}\label{eq:Q1}
\text{for all }0\le p\le p_+, \qquad
\gamma(\sG(n,p)) \ge k
\quad\text{a.a.s.}
\end{equation}
since this is a non-increasing property with respect to the addition of edges.
In fact, a.a.s.\ $\gamma(\sG(n,p_+)) = k$ but we do not prove it now, since we will need a stronger statement to hold.

Now, let $\omega=\omega_n$ be a function tending to infinity sufficiently slowly in order to meet all requirements in the argument. Define
$$
p_-  := p_+ - \frac {2\ \omega \sqrt{p_+} n}{{n \choose 2}} = p_+ \left(1 - \frac {2\ \omega n}{ \sqrt{p_+} {n \choose 2}}\right)=p_+(1+o(1)),
$$
where the last step follows from the fact that $p_+ \ge n^{-1/3+\eps'}$. Since $p_- \sim p_+$, then $p_- \ge n^{-1/3+\eps'/2}$ and $p_-$ is bounded away from 1. Clearly, 
\begin{equation}\label{eq:fpmb}
f(n,k-1,p_-) \le f(n,k-1,p_+) = \frac {1}{p_+ n} < \frac {1}{p_- n}.
\end{equation}
Let $r_-$ be defined as $r$ in~\eqref{eq:rdef} for $p=p_-$. Next we want to show that $r_-=k$ and then that $f(n,k,p_-) \ge \exp(\omega \log n )$. First, using Lemma~\ref{lem:relations}(ii) and the fact that $k=r_+=\Theta(\log n/p_+)$, we get 
\begin{align*}
1-\left(1-p_-\right)^{k-1} &= 1-\Big(1-p_{+} +\Theta \left( \frac {\omega \sqrt{p_+}}{n} \right) \Big)^{k-1} \\
&= 1- (1-p_{+})^{k-1} \left(1+\Theta\left(\tfrac{\omega \sqrt{p_+}}{n}\right)\right)^{k-1} \\
&= 1- (1-p_{+})^{k-1} \left(1+\Theta\left( \tfrac{\omega \log n }{n \sqrt{p_+}}\right)\right) \\
&= 1- (1-p_{+})^{k-1} - \Theta\left( \tfrac{\omega \log^3 n }{n^2 p_+^{3/2} }\right) \\
&= \left(1-   (1-p_{+})^{k-1}\right)  \left(1 - \Theta\left( \tfrac{\omega \log^3 n }{n^2 p_+^{3/2} }\right) \right).
\end{align*}
Hence, 
\begin{align*}
f(n,k-1,p_-) &= f(n,k-1,p_+) \left(1 - \Theta\left( \tfrac{\omega \log^3 n }{n^2 p_+^{3/2} }\right) \right)^{n-k+1} \\
&= f(n,k-1,p_+) \left(1 - \Theta\left( \tfrac{\omega \log^3 n }{n p_+^{3/2} }\right) \right)\\
&\sim f(n,k-1,p_+) = (p_+ n)^{-1} \sim (p_- n)^{-1},
\end{align*}
as $p_+ \ge n^{-1/3+\eps'}$. Combining this with~\eqref{eq:easyratio} and~\eqref{eq:fpmb}, we obtain that $f(n,k,p_-) > 1/(p_-n)$ and $f(n,j,p_-) \le 1/(p_-n)$ for all $j<k$, so $k=r_-$. Now, using Lemma~\ref{lem:relations} again (this time part (iii)), we get
$$
f(n,k,p_-) = f(n,k-1,p_-) \exp(\Theta(\log^2 n)) \sim (p_- n)^{-1} \exp(\Theta(\log^2 n)) \geq \exp(\omega \log n),
$$
as desired. The same argument holds clearly with $n-1$ playing the role of $n$. Therefore, it follows from Theorem~\ref{thm:main_general} that a.a.s.\ $\gamma(\G(n,p_-)) = k$ and $b( \sG(n,p_-) ) = \Theta(\Delta( \sG(n,p_-) )) \ge c p_- n$, for some constant $c=c(\eps)>0$. Let $Q$ be the graph property that we cannot destroy all dominating sets of size $k$ by removing any set of at most $cp_-n$ edges. Clearly, this is a non-decreasing property with respect to adding edges in the graph, so
\begin{equation}\label{eq:Q2}
\text{for all }p_-\le p\le 1, \qquad
\sG(n,p)\text{ satisfies property $Q$ a.a.s.}
\end{equation}

Finally, define
\[
\hat{m} = {n \choose 2} \frac {p_- + p_+}{2} = {n \choose 2} p_+ - \omega n \sqrt{p_+} \sim {n \choose 2} p_+,
\]
where at the last step we use the fact that $p_+>n^{-1/3+\eps}$. Easy manipulations yield
\begin{equation}\label{eq:pplus}
p_+ = \frac{\hat{m} + \omega n \sqrt{p_+}}{\binom{n}{2}}
= \frac{\hat{m} + (\sqrt2+o(1))\omega \sqrt{ \hat m }}{\binom{n}{2}}
\ge \frac{\hat{m} + \omega \sqrt{ \hat m \left(\binom{n}{2}-\hat m \right)/\binom{n}{2}  }}{\binom{n}{2}},
\end{equation}
and similarly
\begin{equation}\label{eq:pminus}
p_- = \frac{\hat{m} - \omega n \sqrt{p_+}}{\binom{n}{2}}
\le \frac{\hat{m} - \omega \sqrt{ \hat m \left(\binom{n}{2}-\hat m \right)/\binom{n}{2}  }}{\binom{n}{2}}.
\end{equation}
In view of~\eqref{eq:Q1}, \eqref{eq:Q2}, \eqref{eq:pplus} and~\eqref{eq:pminus}, we can apply Proposition~1.13 in~\cite{JLR} separately to both the property $Q$ and the property that $\gamma(\G(n,p)) \ge k$, and we conclude that a.a.s.\ $\gamma(\G(n,\hat m)) \ge k$ and $\G(n,\hat m)$ satisfies property $Q$. These two events together imply that $\gamma(\G(n,\hat{m})) = k$ and $b( \G(n,\hat{m}) ) = \Theta(p_-n) = \Theta(m/n)$. The proof is finished.
\end{proof}

\bigskip

Now, we are going to show that Theorem~\ref{thm:main_gnp} can be obtained from Theorem~\ref{thm:main_general}.
\begin{proof}[Proof of Theorem~\ref{thm:main_gnp}]
Let $p=p_n$ be such that $n^{-1/3+\eps}  \le p \le 1-\eps$ for some $\eps > 0$, and let $r=r_n$ be defined as in~\eqref{eq:rdef}. Moreover, suppose there exists a non-increasing non-negative sequence $h=h_n$ such that $p_{n+1}/p_n=1-\Theta(h_n/n)$. Our goal is to show that there exists a positive sequence $\omega=\omega_n\to\infty$ and a dense set $I'\subseteq\nat$ such that
\[
f(n,r,p) \geq \exp(\omega \log n),
\qquad
\text{for $n\in I'$}.
\]
The result will follow immediately from Theorem~\ref{thm:main_general}, and will hold for $I$ defined as in~\eqref{eq:Iprime}. (Note that, since $I'$ is dense, it is straightforward to verify that $I$ must be dense too.)

Throughout the proof, we set $\omega=\omega_n=\log\log n$. Note $h_1 = O(1)$ and so our assumptions on $p$ and $h$ imply that $h_n = O(1)$ and so there exists a universal constant $0 < A_1 < 1$ such that, for every $n\le n'\le 3n$,
\begin{equation}\label{eq:ratiop}
A_1\le \frac{p_{n'}}{p_n} \le 1.
\end{equation}
Given any fixed $j\in\{0,1,2\}$, in view of our assumptions on $p$ and $h$ and by Lemma~\ref{lem:relations}(ii), we have
\begin{align*}
1-\left(1-p_n\right)^{r_{n+1}-j} &= 1-\Big(1-p_{n+1}\big(1+\Theta(h_n/n)\big)\Big)^{r_{n+1}-j}
\\
&= 1-   (1-p_{n+1})^{r_{n+1}-j} \left(1-\Theta\left(\tfrac{p_{n+1} h_n}{n}\right)\right)^{r_{n+1}-j}
\\
&= 1-   (1-p_{n+1})^{r_{n+1}-j} \left(1-\Theta\left( \tfrac{h_n\log n}{n}\right)\right)
\\
&= \left(1-   (1-p_{n+1})^{r_{n+1}-j}\right)  \left(1 + \Theta\left( \frac{h_n\log^3 n}{ n^2p_{n+1}}\right)\right).
\end{align*}
Therefore, 
\begin{align}\label{fchange}
\frac{f(n+1,r_{n+1}-j,p_{n+1})}{f(n,r_{n+1}-j,p_n)} &=
\frac{n+1}{n+1-r_{n+1}+j}
\frac{\left(1-\left(1-p_{n+1}\right)^{r_{n+1}-j}\right)^{n-r_{n+1}+j+1}}
{\left(1-\left(1-p_n\right)^{r_{n+1}-j}\right)^{n-r_{n+1}+j}}
\nonumber \\
&=
\left(1 + \Theta\left(\frac{\log n}{np_n}\right)\right) \left(1-\left(1-p_{n+1}\right)^{r_{n+1}-j}\right)
\left(1 - \Theta\left( \frac{h_n\log^3 n}{ n^2p_{n+1}}\right)\right)^{n-r_{n+1}+j}
\nonumber \\
&=
\left(1 + \Theta\left(\frac{\log n}{np_n}\right)\right) \left(1- \Theta \left(\frac{\log^2 n}{np_{n+1}}\right)\right)
\left(1 - \Theta\left( \frac{h_n\log^3 n}{ n p_{n+1}}\right)\right)
\nonumber \\
&=  1-\Theta\left(\frac{g_n}{np_n}\right),
\end{align}
where $g_n:= \log^2 n +  h_n \log^3 n$. By our assumptions on $h_n$, we have $\log^2 n \leq g_n = O(\log^3 n)$.
In particular, for every $j\in\{0,1,2\}$ and every $n$,
\begin{equation}\label{fchange2}
\exp\left(-C_2\frac{g_n}{np_n}\right) \le \frac{f(n+1,r_{n+1}-j,p_{n+1})}{f(n,r_{n+1}-j,p_n)} \le \exp\left(-C_1\frac{g_n}{np_n}\right),
\end{equation}
for some universal constants $C_2>C_1>0$. From~\eqref{fchange} (with $j=0$) and our assumptions on $p$, we obtain
\[
\frac{(n+1)p_{n+1} f(n+1,r_{n+1},p_{n+1})}{n p_{n} f(n,r_{n+1},p_{n})} = \left(1-\Theta\left(\frac{g_n}{np_n}\right)\right) \left(1+O \left( \frac 1n \right)\right) < 1,
\]
where the last inequality holds for $n$ sufficiently large. This implies that
\[
f(n,r_{n+1},p_{n})  > \frac{(n+1)p_{n+1} f(n+1,r_{n+1},p_{n+1})}{n p_{n} }  > \frac{1}{n p_{n} },
\]
where the last inequality uses the definition of $r_{n+1}$ in~\eqref{eq:rdef}. Hence, $r_{n+1} \ge r_n$ for $n$ large enough, and thus $r$ is a nondecreasing sequence of $n$ except, possibly, for a finite number of terms.
Similarly, from~\eqref{fchange} (with $j=2$) and by Lemma~\ref{lem:relations}(iii),
\begin{eqnarray*}
f(n,r_{n+1}-2,p_{n}) &=& \left(1 +\Theta\left(\tfrac{g_n}{np_n}\right)\right) f(n+1,r_{n+1}-2,p_{n+1}) \\
&\sim& f(n+1,r_{n+1}-2,p_{n+1}) \\
&=& f(n+1,r_{n+1}-1,p_{n+1}) \exp\left(- \Theta(\log^2 n)\right) \\
&\le& \frac{1}{(n+1)p_{n+1}} \exp\left(- \Theta(\log^2 n)\right) < \frac{1}{np_n},
\end{eqnarray*}
for $n$ sufficiently large. Therefore, $r_{n+1}-2<r_n$, or equivalently $r_{n+1} \le r_n +1$, for all but finitely many $n$ (that is, $r$ can increase by at most one). We construct now the set $I'$ as follows:
$$
I' := \{n \in \nat: f(n,r_n,p_{n}) \ge \exp(\omega_n \log n) \}.
$$
Since we want to show that $I'$ contains almost all $n \in \nat$, suppose that $n_1\notin I'$ for some value $n_1\in\nat$. Then we have
\[
1/(n_1p_{n_1}) < f(n_1,r_{n_1},p_{n_1}) < \exp(\omega_{n_1}\log n_1).
\]
Our goal is to show that $n_1$ is followed by an interval of naturals $[n_1,n_2-1]\notin I'$ and then by a much longer interval $[n_2,n_3]\in I'$. We may assume that $n_1$ is sufficiently large, since the limiting density of $I'$ is not affected by ignoring any finite number of naturals.

Let
\[
n_2 = \min\{n>n_1 : r_n>r_{n_1} \text{ or } n=3n_1\}.
\]
Since $r_n=r_{n_1}$ for all $n_1 \le n \le n_2-1$, applying~\eqref{fchange2} to that range (with $j=0$) we get
\begin{eqnarray*}
f(n_2-1,r_{n_2-1},p_{n_2-1}) &\le& f(n_1,r_{n_1},p_{n_1}) \exp\left(-C_1\sum_{n=n_1}^{n_2-2}\frac{g_n}{np_n}\right) \\
&<& \exp\left( \omega_{n_1}\log n_1- \frac {C_1}{3  n_1 p_{n_1}} \sum_{n=n_1}^{n_2-2} g_n \right) \\
&<& \exp\left(  \omega_{n_1}\log n_1 + 1 - \frac {C_1}{3  n_1 p_{n_1}} \sum_{n=n_1}^{n_2-1} g_n \right).
\end{eqnarray*}
On the other hand, by the definition of $r$ (see~\eqref{eq:rdef}), we know that
\begin{equation}\label{eq:fn2minus1}
f(n_2-1,r_{n_2-1},p_{n_2-1}) > \frac {1}{p_{n_2-1}(n_2-1)} \ge \frac {1}{3  p_{n_1} n_1}.
\end{equation}
Hence, it must be the case that, say,
\begin{equation}\label{eq:n2-n1}
\sum_{n=n_1}^{n_2-1} g_n \le n_1 p_{n_1} \omega^{2}_{n_1}\log n_1.
\end{equation}
Since $g_n \ge \log^2 n$ and by our choice of $\omega_n$, it follows that 
\begin{equation}\label{eq:n2n1}
n_2-n_1 \le n_1  \omega^{2}_{n_1} / \log n_1 < n_1.
\end{equation}
As a result, $n_2 \neq 3n_1$,  and so it follows that $r_{n_2} > r_{n_1}$. In fact, since $r$ can increase by at most one, $r_{n_2} = r_{n_2-1} +1$. Now, we get from~\eqref{fchange2}, Lemma~\ref{lem:relations}(iii) and~\eqref{eq:fn2minus1} that, for some small constant $C_3>0$ (possibly depending on $\eps$),
\begin{eqnarray}
f(n_2,r_{n_2},p_{n_2}) & \ge &  C_3 f(n_2-1,r_{n_2},p_{n_2-1}) = C_3 f(n_2-1,r_{n_2-1}+1,p_{n_2-1})
\notag\\
& \ge & C_3 f(n_2-1,r_{n_2-1},p_{n_2-1}) \exp \left( C_3 \log^2 n_2 \right)
\notag\\
& \ge & \frac {C_3/3}{p_{n_1} n_1}  \exp \left( C_3 \log^2 n_2 \right)
\notag\\
& \ge & \exp \left(  (C_3/2)  \log^2   n_2  \right)   \ge \exp \left( \omega_{n_2} \log  n_2 \right).
\label{eq:fn2}
\end{eqnarray}
As a result, $n_2$ belongs to $I'$. 

Let
$$
n_3 = \min \{ n > n_2 : f(n,r_{n_2},p_n) < \exp(2\omega_{n_2}\log n_2) \text{ or } n = 3n_2\}.
$$
Note that if $f(n,r_{n_2},p_n) \ge \exp(2\omega_{n_2}\log n_2)$ for some $n_2 < n \le 3n_2$, then 
$$
f(n,r_{n_2},p_n) \ge \exp(\omega_{n}\log n) > 1/(p_n n).
$$ 
Hence, $r_n = r_{n_2}$ and, more importantly, $n \in I'$. If $n_3 = 3n_2$, then we are done, since the interval $[n_2, n_3]$ is longer than $[n_1, n_2-1]$ by at least a $\log n_1/\omega^2_{n_1}$ factor (see second step in~\eqref{eq:n2n1}). Hence, we may assume that $f(n_3,r_{n_2},p_{n_3}) < \exp(2\omega_{n_2}\log n_2)$. Applying~\eqref{fchange2} one more time and by the second last step of~\eqref{eq:fn2}, we get 
\begin{eqnarray*}
f(n_3,r_{n_2},p_{n_3}) &\ge& f(n_2,r_{n_2},p_{n_2}) \exp \left(-C_2\sum_{n=n_2}^{n_3-1}\frac{g_n}{np_n}\right) \\
&\ge& \exp\left( (C_3/2) \log^2 n_2 - \frac {C_2}{A_1 n_2 p_{n_2}} \sum_{n=n_2}^{n_3-1} g_n \right) \\
&\ge& \exp\left( (C_3/2) \log^2 n_2 - \frac {C_2}{A_1^2 n_1 p_{n_1}} \sum_{n=n_2}^{n_3-1} g_n \right),\end{eqnarray*}
which is at least $\exp(2\omega_{n_2}\log n_2)$, if, say, $\sum_{n=n_2}^{n_3-1} g_n \le n_1 p_{n_1} \log^2 n_1 / \omega_{n_1}.$ Consequently,
\begin{equation}\label{eq:n-n2}
\sum_{n=n_2}^{n_3-1} g_n > \frac {n_1 p_{n_1} \log^2 n_1} { \omega_{n_1} }.
\end{equation}
Finally, note that $h_n$ is non-increasing and $n_3-n_1 \le 8n_1$, so $g_n \sim g_{n_1}$ for any $n_1 \le n \le n_3$ and, as a result,
$$
\max\{g_n : n_2 \le n \le n_3 - 1\}< C \cdot \min\{g_n : n_1 \le n \le n_2 -1\}
$$
for some universal constant $C$. Combining this observation together with~\eqref{eq:n2-n1} and~\eqref{eq:n-n2}, it immediately follows that
$$
\frac {n_2 - n_1}{n_3-n_2+1} \le  \frac{\omega_{n_1}^{3}}{\log n_1}.
$$
Putting everything together, given any $n_1\notin I'$ sufficiently large, we obtained $n_2$ and $n_3$ such that
\[
[n_2,n_3]\subseteq I'
\qquad\text{and}\qquad
n_3-n_2 + 1\ge \frac{\log n_1}{\omega_{n_1}^{3}} (n_2-n_1).
\]
This proves that $I'$ is dense as required, and the proof of the theorem is finished. 
\end{proof}

\bigskip

Note that the lengths of the intervals $[n_1,n_2-1]$ and $[n_2,n_3]$ in the proof of Theorem~\ref{thm:main_gnp} depend on the value of $p_n$. This is not an artifact of the proof, but rather reflecting the fact that for different values of $p_n$ these lengths are indeed different: for $p_n=n^{-1/4}$, we get that $r_{2n}-r_n=\Theta(n^{1/4}\log n)$, and thus, on average, after $\Theta(\frac{np}{\log n})$ integers the value of $r$ increases by $1$. On the other hand, for $p_n=\frac{1}{\log \log n}$, we get that $r_{2n}-r_n=\Theta(\log \log n)$, and thus, on average, after $\Theta(np)$ integers the value of $r$ increases by $1$.

%
%
\section{Second moment ingredients}\label{sec:second}

For a given function $p=p_n$, let $f(n,k,p)$ and $r=r_n$ be defined as in~\eqref{eq:fdef} and~\eqref{eq:rdef}, respectively. Throughout this section, we suppose that, given our choice of $p$, there exists some infinite set $I' \subseteq\nat$ satisfying~\eqref{eq:condI} for a given function $\omega = \omega_n \to\infty$, and restrict all our asymptotic statements to $n\in I'$. For simplicity, we also write $X$ instead of $X_r$ and $\cD$ instead of $\cD_r$. For each $i\in\{0,1,\ldots,r\}$, let $W_i$ be the random variable counting the number of ordered pairs $D,D'\in\cD$ in $\sG(n,p)$ with $|D\cap D'|=i$. One of the key ingredients in our analysis is to estimate the variance of $X$ and other related random variables defined later in the paper. To do so, we will use several bounds on $\ex W_i$ that are stated in~Proposition~\ref{prop:variancekey} below. In fact, the variance of $X_{r+1}$ was already studied in~\cite{GLS} and~\cite{WG}, and we follow some of the ideas in their computations, but we need a more accurate estimation of the error terms involved. Also, the aforementioned papers deal with $X_{r+1}$ instead of $X_r$, since they make use of the fact that $\ex X_{r+1}=\exp(\Theta(\log^2n))$. In our case, this fact is replaced by our assumption~\eqref{eq:condI}.

The following lemma uses some of the computations in~\cite{GLS}, and will prepare us for~Proposition~\ref{prop:variancekey}. Given two sets of vertices $D,D'$ of size $r$ with $|D \cap D'|=i$, let $P_i$ denote the probability that $D,D'$ dominate each other in $\sG(n,p)$ (i.e.,~every vertex in $D$ has a neighbour in $D'$ and vice versa).
\begin{lem}\label{lem:large_range} 
Given a constant $\epsilon>0$, suppose that $\log^2 n / \sqrt{n} \ll p \leq 1-\eps$ and condition~\eqref{eq:condI} holds for some infinite set $I'\subseteq\nat$ and some function $\omega =\omega_n \to\infty$, where $f(n,k,p)$ and $r$ are defined as in~\eqref{eq:fdef} and~\eqref{eq:rdef}. Then, for each $0.9r\le i\le r$,
\[
\frac{r\ex W_i/P_i}{\ex W_1/P_1}  \le \exp\left(-(\omega/2 )\log n \right)  \qquad\text{(for $n\in I'$)}.
\]
\end{lem}
\begin{proof}[Sketch of proof]
We follow some of the computations in Section~3.1 of~\cite{GLS}. In that paper, their choice of $r$ corresponds to our $r+1$, and their calculations assume $p=o(1)$, but everything we use here remains valid in our setting. 
First, note that
\begin{equation}\label{eq:EWi}
\ex W_i = \frac{n!}{i!(r-i)!^2(n-2r+i)!} \left(  1-(1-p)^i + (1-p)^i(1-(1-p)^{r-i})^2 \right)^{n-2r+i}P_i.
\end{equation}
Observe that $\ex W_i/P_i$ corresponds exactly to $f(i)$ in~\cite{GLS}. By adapting~(3.8) in~\cite{GLS} to our notation and using our assumption~\eqref{eq:condI}, we get
\begin{align*}
\frac{r\ex W_i/P_i}{\ex W_1/P_1} &\le \frac{(1+o(1))}{\ex X} \frac{n}{r} \binom{r}{i}\binom{n-r}{r-i}
\left(1+\frac{(1-p)^r\left(1-(1-p)^{r-i}\right)}{1-2(1-p)^r+(1-p)^{2r-i}}\right)^{-(n-r)}
\\
&\le \exp\left( - \omega \log n + \log n + 2(r-i)\log n - n
(1-p)^r\left(1-(1-p)^{r-i}\right)(1+o(1)) \right).
\end{align*}
(The term $-\omega \log n$ in the exponent above corresponds to $-(1+o(1))\log^2(pn)$ in~\cite{GLS}, because of their different choice of $r$.)
Moreover, (3.9) in~\cite{GLS} gives that
\[
2(r-i)\log n - \frac{n}{2} (1-p)^r\left(1-(1-p)^{r-i}\right)  \le 0,
\]
and therefore
\[
\frac{r\ex W_i/P_i}{\ex W_1/P_1} \le \exp\left( - \omega \log n + \log n  \right) \le \exp\left( - (\omega/2) \log n \right).\qedhere
\]
\end{proof}

\bigskip
Before we proceed, we need one more definition. Given a constant $\eps>0$ and for $i\in\{1, 2, \ldots,r\}$, let
\begin{equation}\label{eq:Qidef}
Q_i = \sum_{j=0}^{\min\{i-1,L-1\}} \pr(\Bin(i-1,p)=j) (\pr(\Bin(r-i,p)<L-j))^2,
\end{equation}
where $L = \lfloor \varrho p r \rfloor$ with $\varrho=\varepsilon^2$.  The following proposition will be central for estimating the variance of several random variables.
\begin{prop}\label{prop:variancekey}
Given a constant $\varepsilon > 0$, assume that $\log^2 n / \sqrt{n} \ll p \leq 1-\varepsilon$ and condition~\eqref{eq:condI} is satisfied for some infinite set $I'\subseteq\nat$. Then, the following holds for $\sG(n,p)$ with $n$ restricted to $I'$:
\begin{itemize}
\item[(i)]
\[
\frac{\ex W_0}{(\ex X)^2} \le 1 + \Theta\left(\frac{\log^3n}{p^2n}\right)
\qquad\text{and}\qquad
\frac{\ex W_1}{(\ex X)^2} \le \frac{r^2}{n} \left( 1 + \Theta\left(\frac{\log^4n}{pn} + \frac{\log^3n}{p^2n}\right) \right);
\]
\item[(ii)] 
$$
 \sum_{i=1}^r i\ex W_i \le \frac{r^2}{n} \left( 1 + \Theta\left(\frac{\log^4n}{pn} + \frac{\log^3n}{p^2n}\right) \right) (\ex X)^2;
$$
\item[(iii)]
\[
\sum_{i=1}^r i Q_i \ex W_i \le Q_1\frac{r^2}{n}\left( 1 + \Theta\left(\frac{\log^4n}{pn} + \frac{\log^3n}{p^2n}\right) \right) (\ex X)^2.
\]
\end{itemize}
\end{prop}
\begin{proof}
Denoting by $P_i$ is, as above, the probability that $D,D'$ with intersection of size $i$ dominate each other, note first that from~\eqref{eq:EWi} we obtain in particular,
\begin{equation}\label{eq:EW0}
\ex W_0 = \frac{n!}{r!^2(n-2r)!} \left(1-(1-p)^{r}\right)^{2(n-2r)}P_0,
\end{equation}
and
\begin{equation}\label{eq:EW1}
\ex W_1 = \frac{n!}{(r-1)!^2(n-2r+1)!} \left(  p + (1-p)(1-(1-p)^{r-1})^2 \right)^{n-2r+1}P_1.
\end{equation}
Also, recall that
\begin{equation}\label{eq:EX2}
(\ex X)^2 = f(n,r,p)^2 = \left(\frac{n!}{r!(n-r)!} \left(1-(1-p)^{r}\right)^{n-r}\right)^2.
\end{equation}

Using~\eqref{eq:EW0}, \eqref{eq:EX2} and Lemma~\ref{lem:relations}(ii), we can easily bound the ratio
\[
\frac{\ex W_0}{(\ex X)^2} = \frac{[n-r]_r}{[n]_r} \big(1-(1-p)^{r}\big)^{-2r}P_0
\le
 \big(1-(1-p)^{r}\big)^{-2r}
= 1 + \Theta(r(1-p)^{r})
= 1 + \Theta\left(\frac{\log^3n}{p^2n}\right).
\]
Moreover, from~\eqref{eq:EW1}, \eqref{eq:EX2},
Lemma~\ref{lem:relations}(ii) and the fact that $p \gg \log^2n/\sqrt n$,
\begin{align*}
\frac{\ex W_1}{(\ex X)^2} &= \frac{r^2[n-r]_{r-1}}{[n]_r} \frac{\left(  p + (1-p)(1-(1-p)^{r-1})^2 \right)^{n-2r+1}} { \left( \left(1-(1-p)^{r}\right)^{n-r}\right)^2} P_1
\\
&= \frac{r^2[n-r]_{r-1}}{[n]_r} \frac{\left( 1 - 2(1-p)^{r} + (1-p)^{2r-1} \right)^{n-2r+1}} { \left( \left(1-(1-p)^{r}\right)^{n-r}\right)^2} P_1
\\
&= \frac{r^2[n-r]_{r-1}}{[n]_r} \left( 1 + \frac{ (1-p)^{2r-1} - (1-p)^{2r} )} {\left(1-(1-p)^{r}\right)^2}\right)^{n-r} \left( 1 - 2(1-p)^{r} + (1-p)^{2r-1} \right)^{-r+1} P_1
\\
&= \frac{r^2[n-r]_{r-1}}{n[n-1]_{r-1}} \left( 1 + \frac{ p(1-p)^{2r-1}} {\left(1-(1-p)^{r}\right)^2}\right)^{n-r} \left( 1 - 2(1-p)^{r} + (1-p)^{2r-1} \right)^{-r+1} P_1
\\
&= \frac{r^2}{n} \left(1 -  \Theta\left(\frac{r^2}{n}\right)\right) \left(1 +  \Theta\left(pn(1-p)^{2r}\right)\right) \left( 1 + \Theta\left(r(1-p)^{r}\right) \right) P_1
\\
&= \frac{r^2}{n} \left(1 -  \Theta\left(\frac{\log^2n}{p^2n}\right)\right) \left(1 + \Theta\left(\frac{\log^4n}{pn}\right)\right) \left( 1 + \Theta\left(\frac{\log^3n}{p^2n}\right) \right) P_1
\\
&\le \frac{r^2}{n} \left( 1 + \Theta\left(\frac{\log^4n}{pn} + \frac{\log^3n}{p^2n}\right) \right).
\end{align*}
This proves part~(i). Note that, in fact, we get something slightly stronger, namely
\begin{equation}\label{eq:stronger}
\frac{\ex W_1/P_1}{(\ex X)^2} \le \frac{r^2}{n} \left( 1 + \Theta\left(\frac{\log^4n}{pn} + \frac{\log^3n}{p^2n}\right) \right).
\end{equation}
For $i$ not too close to $r$, say $1\le i\le r-3\log\log n/p$, we have
\begin{align}
\frac{\ex W_{i+1}/P_{i+1}}{\ex W_i/P_i} &= \frac{(r-i)^2}{(i+1)(n-2r+i+1)} \frac{\left(  1-(1-p)^{i+1} + (1-p)^{i+1}(1-(1-p)^{r-i-1})^2 \right)^{n-2r+i+1}}{\left(  1-(1-p)^i + (1-p)^i(1-(1-p)^{r-i})^2 \right)^{n-2r+i}}
\notag\\
&= \frac{(r-i)^2}{(i+1)(n-2r+i+1)} \left(\frac{ 1 - 2(1-p)^r+(1-p)^{2r-i-1} }{ 1 - 2(1-p)^{r} + (1-p)^{2r-i} }\right)^{n-2r+i} \left(  1 - 2(1-p)^r+(1-p)^{2r-i-1} \right)
\notag\\
&= \frac{(r-i)^2}{(i+1)(n-2r+i+1)} \left(1+\frac{p(1-p)^{2r-i-1}}{ 1 - 2(1-p)^{r} + (1-p)^{2r-i} }\right)^{n-2r+i} \left(  1 - 2(1-p)^r+(1-p)^{2r-i-1} \right)
\notag\\
&\le \frac{r^2}{n-2r} \left(1+O\left(pn(1-p)^{r+3\log\log n/p}\right)\right) \left(  1 - \Theta\left((1-p)^r\right) \right)
\notag\\
&\le \frac{r^2}{n-2r} \left(1+O\left(e^{-3\log\log n}\log^{2}n\right)\right) \left(  1 - \Theta\left(\frac{\log^2n}{pn}\right) \right)
\notag\\
&= \frac{r^2}{n}(1+o(1)) = O(\log^2 n/(p^2 n)).
\label{eq:ratioEWi}
\end{align}
On the other hand, consider now  $r-3\log\log n/p \le i \le r$. Since this range is eventually included in the range $0.9r \le i \le r$ then, by Lemma~\ref{lem:large_range},
\begin{equation}\label{eq:germans_eq}
\frac{r\ex W_i/P_i}{\ex W_1/P_1} \le \exp\left(- (\omega/2) \log n \right). 
\end{equation}    
Now, note that for $i \geq 1$ we have
$$
\frac{(i+1)\ex W_{i+1}/P_{i+1}}{i \ex W_i/P_i} \leq 2\frac{\ex W_{i+1}/P_{i+1}}{\ex W_i/P_i}.
$$ 
Combining this with~\eqref{eq:stronger}, \eqref{eq:ratioEWi} and~\eqref{eq:germans_eq},
\begin{align*}
\sum_{i=1}^r i \ex W_i &\leq \sum_{i=1}^r i\ex W_i/P_i =  \frac{ \ex W_1}{P_1} \left(1+O\left(\frac{\log^2 n}{p^2 n}\right)\right) + \frac{\ex W_1}{P_1} O\left(\log \log n/p\right)\exp\left(  - (\omega/2) \log n \right) \\
&= \frac{\ex W_1}{P_1} \left(1+O\left( \frac{\log^2 n}{p^2n} \right)\right) \\
&\leq (\ex X)^2\frac{r^2}{n} \left( 1 + \Theta\left(\frac{\log^4n}{pn} + \frac{\log^3n}{p^2n}\right)\right)\left(1+O\left( \frac{\log^2 n}{p^2n}\right)\right)  \\
 &= (\ex X)^2\frac{r^2}{n} \left( 1 + \Theta\left(\frac{\log^4n}{pn} + \frac{\log^3n}{p^2n}\right)\right),
\end{align*}
and part (ii) follows. 

For part (iii), observe first that there exists some $C =C(\eps) > 0$ such that if $i \ge C \log n$, then $i > 2L$.  Hence, for $C \log n \leq i \leq r-3\log\log n/p$, substituting $B_{a,b}$ for $\pr(\Bin(a,p)=b)$, we get
\begin{eqnarray*}
\frac{(i+1)Q_{i+1}}{iQ_i} &=& \frac{(i+1)\sum_{j=0}^{L-1} \pr(\Bin(i,p)=j) (\pr(\Bin(r-i-1,p)<L-j))^2}{i\sum_{j=0}^{L-1} \pr(\Bin(i-1,p)=j) (\pr(\Bin(r-i,p)<L-j))^2} \\
&=& \frac{(i+1) \sum_{j=0}^{L-1}\sum_{k=0}^{L-j-1}\sum_{k'=0}^{L-j-1} \frac{i}{i-j} \frac{r-i-k}{r-i} \frac{r-i-k'}{r-i} \frac{1}{1-p} B_{i-1,j} B_{r-i,k} B_{r-i,k'}} {i\sum_{j=0}^{L-1} \sum_{k=0}^{L-j-1}\sum_{k'=0}^{L-j-1} B_{i-1,j} B_{r-i,k} B_{r-i,k'} } \\
&\le& \frac {i+1}{i} \frac {i}{i-L+1} \frac {1}{1-p} = O(1).
\end{eqnarray*}
Similarly, for $L \le i < C \log n$ we have 
$$
 \frac{(i+1)Q_{i+1}}{iQ_i}= O(\log n).
$$
On the other hand, for $1 \leq i \leq L-1$, 
\begin{eqnarray*}
   \frac{(i+1)Q_{i+1}}{iQ_i} &=& \frac{(i+1) \sum_{j=0}^{i-1}\sum_{k=0}^{L-j-1}\sum_{k'=0}^{L-j-1} \frac{i}{i-j}\frac{r-i-k}{r-i} \frac{r-i-k'}{r-i} \frac{1}{1-p}B_{i-1,j} B_{r-i,k} B_{r-i,k'}} {i\sum_{j=0}^{i-1} \sum_{k=0}^{L-j-1}\sum_{k'=0}^{L-j-1} B_{i-1,j} B_{r-i,k} B_{r-i,k'} } \\
&&   +\frac{(i+1)p^i \sum_{k=0}^{L-i-1}\sum_{k'=0}^{L-i-1} B_{r-i-1,k} B_{r-i-1,k'} } {i\sum_{j=0}^{i-1} \sum_{k=0}^{L-j-1}\sum_{k'=0}^{L-j-1} B_{i-1,j}  B_{r-i,k} B_{r-i,k'} }\\
&=& O(\log n)+\frac{(i+1)p^i (\frac{1}{1-p})^2\sum_{k=0}^{L-i-1}\sum_{k'=0}^{L-i-1} B_{r-i,k} B_{r-i,k'} } {ip^{i-1} \sum_{k=0}^{L-i}\sum_{k'=0}^{L-i} B_{r-i,k} B_{r-i,k'} } \\
&=& O(\log n).
\end{eqnarray*}
Finally, for $r-3\log\log n/p \leq i \leq r$, by Lemma~\ref{lem:large_range}, since $Q_i \leq 1$, and by Chernoff's bound (see~\eqref{eq:chernoffstrong}), 
\begin{eqnarray*}
\frac{r Q_i \ex W_i/P_i}{Q_1 \ex W_1/P_1} &\le & \frac{ Q_i}{Q_1} \exp\left(- (\omega/2)\log n \right) \\
& \leq &\frac{1}{(\pr(\Bin(r-1,p) < L))^2}\exp\left(- (\omega/2)\log n \right) \\
& \leq & \exp\left(((1-\varrho)^2pr/2)-  (\omega/2)\log n \right) \\
& \leq & \exp\left(O(\log n)-  (\omega/2)\log n \right) \le \exp\left(- (\omega/3)\log n \right),
\end{eqnarray*}
where the last inequality follows from $p \gg  \log^2 n/\sqrt{n}$. Combining all bounds, 
\begin{eqnarray*}
\sum_{i=1}^r i Q_i \ex W_i &\leq& \sum_{i=1}^r i Q_i \ex W_i/P_i =  \frac{ \ex W_1 Q_1}{P_1} \left(1+O\left(\frac{\log^3 n}{p^2 n}\right)\right) \\
 &=& Q_1(\ex X)^2\frac{r^2}{n} \left( 1 + O\left(\frac{\log^3n}{p^2n}\right)+O\left(\frac{\log^4n}{pn}\right)\right),
 \end{eqnarray*}
 and (iii) follows. The proof of the proposition is finished.
\end{proof}
As an immediate consequence of this proposition, we can bound the variance of $X=X_r$ (which is also done for $X_{r+1}$ in~\cite{GLS} and~\cite{WG}), and obtain the following result.
\begin{cor}\label{cor:small_variance}
Given a constant $\eps>0$, assume that $\log^2n/\sqrt n \ll p \le 1-\eps$ and condition~\eqref{eq:condI} is satisfied for some infinite set $I'\subseteq\nat$. Then a.a.s.\ $X \sim f(n,r,p)$ (for $n\in I'$). Consequently, a.a.s.\ $\gamma(\sG(n,p))=r$ (for $n\in I'$).
\end{cor}
\begin{proof}
From Proposition~\ref{prop:variancekey}~(i) and~(ii), we get
\[
\ex W_0 \le (1+o(1)) (\ex X)^2
\qquad\text{and}\qquad
\sum_{i=1}^r \ex W_i \le \sum_{i=1}^r i\ex W_i \le O(r^2/n) (\ex X)^2 = o (\ex X)^2,
\]
where we used that $r^2/n=o(1)$ by Lemma~\ref{lem:relations}(ii). Therefore,
\[
\var X = \ex(X^2) - (\ex X)^2 = \sum_{i=0}^r \ex W_i - (\ex X)^2 = o(\ex X)^2,
\]
and thus, by Chebyshev's inequality, we conclude that $X \sim \ex X = f(n,r,p)\to\infty$ a.a.s.\ for $n \in I'$. The second claim in the statement follows immediately from the fact that $\ex X_{r-1}=f(n,r-1,p)=o(1)$ (from the definition of $r$ in~\eqref{eq:rdef}).
\end{proof}

\bigskip

Before we state the next lemma, we need one more definition. For a given vertex $v$, let $Z_v$ be the random variable counting the number of dominating sets of size $r$ containing vertex $v$. We will use Proposition~\ref{prop:variancekey} to prove the following observation. 

\begin{lem}\label{lem:vertex}
Given a constant $\varepsilon > 0$,  assume that $\log^2 n / \sqrt{n} \ll p \leq 1-\varepsilon$ and condition~\eqref{eq:condI} is satisfied for some infinite set $I'\subseteq\nat$. Then, the following holds for $\sG(n,p)$ and any vertex $v \in [n]$:
$$
\ex Z_v = \frac{r}{n} \ \ex X \qquad \text{and} \quad \text{(for $n\in I'$)} \qquad \var Z_v = \Theta\left(\frac{\log^4n}{pn} + \frac{\log^3n}{p^2n}\right) (\ex Z_v)^2.
$$
\end{lem}
\begin{proof}
First note that
\[
\sum_{v\in[n]} Z_v = r X,
\]
as both sides count dominating sets in $\cD$ with one vertex marked. So $\ex Z_v = \frac{r}{n} \ex X$ by linearity of expectation and since all $Z_v$ have the same distribution, and the first part holds. Similarly,
\[
\sum_{v\in[n]} {Z_v}^2 = \sum_{i=1}^r i W_i,
\]
since both sides count pairs of dominating sets $D,D'\in\cD$ with one marked vertex in the intersection. Therefore,
\[
\ex({Z_v}^2) = \frac1n \sum_{i=1}^r i \ex W_i \le \frac{r^2}{n^2} \left( 1 + h \right) (\ex X)^2 = \left( 1 + h \right) (\ex Z_v)^2,
\]
for some $h =\Theta(\frac{\log^4n}{pn} + \frac{\log^3n}{p^2n})$, by Proposition~\ref{prop:variancekey}. The bound on the variance in the statement follows immediately, and the proof of the lemma is finished.
\end{proof}

\section{Proof of Theorem~\ref{thm:main_general}}\label{sec:theorem}

In order to prove our main result, we first analyze the effect that removing one edge has on the number of dominating sets of smallest size. Given $p=p_n$, recall the definitions of $f(n,k,p)$ and $r$ in~\eqref{eq:fdef} and~\eqref{eq:rdef}. Also recall $X=X_r$ and $\cD=\cD_r$. Let $G=(V,E)$ be a random outcome of $\sG(n,p)$. Throughout this section, a {\bf pair} $uv$ always refers to a pair of different vertices $u,v\in V$ (a pair $uv$ may or may not be an edge in $E$). Similarly, a {\bf directed pair} $\overrightarrow{uv}$ refers to the corresponding ordered pair of vertices (so $uv=vu$ but $\overrightarrow{uv}\ne\overrightarrow{vu}$). Given a pair $uv$, let $\widehat\cD_{uv}$ be the set of dominating sets of size $r$ of the graph $G+uv = (V,E\cup\{uv\})$. Given a directed pair $\overrightarrow{uv}$ and $j \in [r]$, let
\[
\widehat\cD_{j,\overrightarrow{uv}}
= \big\{D\in\widehat\cD_{uv} : v\in D,u\notin D, |\widehat N(u)\cap D|=j\big\},
\]
where $\widehat N(u)$ denotes the set of vertices adjacent to $u$ in $G+uv$.
Define the \textbf{damage} of $\overrightarrow{uv}$ to be
\[
Z_{\overrightarrow{uv}} = \sum_{j=1}^r \frac{| \widehat \cD_{j,\overrightarrow{uv}} |}{j},
\]
and the \textbf{damage} of the corresponding pair $uv$ to be $Z_{uv}=Z_{\overrightarrow{uv}}+Z_{\overrightarrow{vu}}$. Finally, the \textbf{damage} of a set of pairs $A$ is $Z_A=\sum_{e\in A}Z_e$. We will see that this notion constitutes a convenient upper bound on the number of dominating sets of size $r$ destroyed by removing a set of pairs $A$ from the edge set. Let $Y_A$ be the number of dominating sets in $\cD$ that are not dominating anymore after deleting a set of pairs $A$ from $E$, that is, the number of dominating sets of size $r$ of $G$ but not of $G-A = (V, E\setminus A)$. (Note that the definitions of $Z_A$ and $Y_A$ do not require $A\subseteq E$, but in the next observation we do.)

\begin{obs}\label{obs:Y}
Assuming $\gamma(G)=r$, clearly, one strategy to prove a lower bound $b(G)>a$ is to show that $Y_A < X$ for all sets of edges $A\subseteq E$ of size $a$, so the removal of any $a$ edges of $G$ cannot destroy all dominating sets of minimal size.
Unfortunately, $Y_A$ is not easy to compute, since in general $Y_A \ne \sum_{e\in A}Y_e$.
Hence, our notion of damage turns useful in view of the following deterministic result.
\end{obs}
\begin{lem}\label{lem:YZ}
For every set $A$ of pairs (not necessarily $A\subseteq E$), $Y_A\le Z_A$. 
\end{lem}
\begin{proof}
The proof is straightforward. Let $D \in \cD$ be a dominating set of $G$ of size $r$ contributing to $Y_A$. Since $D$ fails to dominate the rest of the graph $G-A$, there must be some vertex $u\notin D$ (but, of course, adjacent to some vertex in $D$) such that all $|N(u)\cap D|$ edges connecting $u$ and $D$ in $G$ belong to $A$ (and thus are removed). Each of the corresponding directed pairs $\overrightarrow{uv}$ (for $v\in N(u)\cap D$) contributes $1/|N(u)\cap D|$ to the total damage.
\end{proof}
In order to bound $Y_A$, by the previous lemma, it suffices to estimate $Z_{uv}=Z_{\overrightarrow{uv}}+Z_{\overrightarrow{vu}}$ and sum over all pairs $uv$ in $A$.
It is convenient for our analysis to split the damage $Z_{\overrightarrow{uv}}$ of a directed pair $\overrightarrow{uv}$ into its {\bf light damage}
\[
Z'_{\overrightarrow{uv}} = \sum_{j=L+1}^{r} \frac{|  \widehat\cD_{j,\overrightarrow{uv}}  |}{j},
\]
and its {\bf heavy damage}  
\[
Z''_{\overrightarrow{uv}} = Z_{\overrightarrow{uv}} - Z'_{\overrightarrow{uv}} = \sum_{j=1}^{L}  \frac{|  \widehat\cD_{j,\overrightarrow{uv}}  |}{j}.
\]
(Recall that $L = \lfloor \varrho p r \rfloor$ with $\varrho=\varepsilon^2.$) Similarly as before, the {\bf light damage} of a pair $uv$ is $Z'_{uv}=Z'_{\overrightarrow{uv}}+Z'_{\overrightarrow{vu}}$, and its {\bf heavy damage} is $Z^{''}_{uv}=Z^{''}_{\overrightarrow{uv}}+Z^{''}_{\overrightarrow{vu}}$.  For a given set of pairs $A$, its \textbf{light damage} is $Z'_A=\sum_{e \in A} Z'_e$ and its \textbf{heavy damage} is $Z^{''}_A=\sum_{e \in A} Z^{''}_e$.

We will now estimate the first and second moments of some of the random variables described above. Given any $\overrightarrow{uv}$, we can easily estimate $\ex Z_{\overrightarrow{uv}}$ by summing the probability that a given $D\in\widehat\cD_{j,\overrightarrow{uv}}$ appears in $G+uv$, weighted by $1/j$, over all possible choices of $D$.
\begin{align}
\ex Z_{\overrightarrow{uv}} &= \sum_{j=1}^r \frac{1}{j} \binom{n-2}{r-1} \Big(1-(1-p)^r\Big)^{n-r-1}\binom{r-1}{j-1}p^{j-1}(1-p)^{r-j}
\nonumber \\
&= \frac{r(n-r)}{n(n-1)} \Big(1-(1-p)^r\Big)^{-1} \ex X \sum_{j=1}^r \frac{1}{pr} \binom{r}{j}p^j(1-p)^{r-j}
\label{eq:zuv} \\
&= \frac{(n-r)}{pn(n-1)} \ex X \sim \frac{\ex X}{pn}.
\nonumber
\end{align}
For $Z''_{\overrightarrow{uv}}$ we get better bounds: 
\begin{lem}\label{lem:exp_zuv}
Given any constant $\varepsilon > 0$ sufficiently small, assume that $\log^2/\sqrt n \ll  p \le 1-\varepsilon$. Then, for $n$ sufficiently large and for any $\overrightarrow{uv}$, 
\begin{equation*}
\ex Z''_{\overrightarrow{uv}}
\le
\begin{cases}
\dfrac{\ex X}{(pn)^{2-\varepsilon/2}}, 
& \text{if } p=o(1), 
\\\\
\dfrac{\ex X}{(pn)^{1+\varepsilon^2}},
& p=\Theta(1) \mbox{ and } p \le 1-\eps.
\end{cases}
\end{equation*}
\end{lem}
\begin{proof}
Arguing as in~\eqref{eq:zuv} and by Lemma~\ref{lem:relations}(ii), we have 
\begin{align*}
\ex Z''_{\overrightarrow{uv}} &= \frac{r(n-r)}{n(n-1)} \Big(1-(1-p)^r\Big)^{-1} \ex X \sum_{j=1}^{L} \frac{1}{pr} \binom{r}{j}p^j(1-p)^{r-j}
\\
& = \frac{(n-r)}{pn(n-1)} \ex X \frac{ \pr (1 \le \Bin(r,p) \le L) }{1-(1-p)^r} 
\\
& \leq \frac{ 1+o(1) }{pn} \ex X \; \pr (\Bin(r,p) \leq L).
\end{align*}
By the stronger version of Chernoff's bound given in~\eqref{eq:chernoffstrong}, writing $\varrho'=\varrho-\varrho \log \varrho$,
$$
\pr (\Bin(r,p) \leq L) = \pr (\Bin(r,p) \leq \varrho pr) \leq \exp(-rp\varphi(\varrho-1)) \leq \exp(-rp(1-\varrho')).
$$
 Note that $\varrho'$ gets small when $\varrho$ does, even if at a slower rate.
If $p\to0$, using $\p=\log(1/(1-p))$, by Lemma~\ref{lem:relations}(i) we have 
\[
\exp(-(1-\varrho')rp) \le \left(\frac{\p n (1+o(1))}{\log^2 n}\right)^{-(1-\varrho')p/\p}
\le \left(p n\right)^{-(1-2\varrho')}.
\]
Now, by our choice of $\varrho=\varepsilon^2$ and using the fact that $2\varepsilon \log \varepsilon \to 0$ as $\varepsilon \to 0$, we have $1-2\varrho'=1-2\varrho+2\varrho \log \varrho > 1-\varepsilon/2$, and the statement follows in this case.

If $p=\Theta(1)$ with $p$ bounded away from $1$, we have 
\[
\exp(-(1-\varrho')rp) \le \left(p n\right)^{-\varrho'},
\]
where we assumed that $\varepsilon$ (and thus $\varrho$) was chosen to be small enough so that the following holds: $\varrho' < (1-\varrho')p/\log(1/(1-p))$ (note that $p$ close to $1$ forces a small $\varrho$, and therefore a small $\varepsilon$). The desired statement follows since $\varrho' > \varrho = \varepsilon^2$.
\end{proof}

In order to bound the variance of $Z''_{\overrightarrow{uv}}$, we will need to use that~\eqref{eq:condI} holds for two consecutive values $n-1$ and $n$. Therefore, we assume that there exist infinitely many such pairs of values, and restrict asymptotics to all $n$ such that both $n-1$ and $n$ satisfy~\eqref{eq:condI}.

\begin{lem}\label{lem:var2}
Given a constant $\varepsilon > 0$, assume that $\log^2 n / \sqrt{n} \ll p \leq 1-\varepsilon$. Moreover, suppose there exist infinite sets $I'\subseteq I\subseteq\nat$ satisfying~\eqref{eq:condI} and~\eqref{eq:Iprime}, and restrict asymptotics to $n\in I'$. Then,
\[
\var Z''_{\overrightarrow{uv}} = O\left(\frac{\log^4n}{pn} + \frac{\log^3n}{p^2n}\right)(\ex Z''_{\overrightarrow{uv}})^2. 
\]
\end{lem}
\begin{proof}
First, observe that $\sG(n,p)-u$ is distributed as $\sG(n-1,p)$, and this is independent of the edges emanating from $u$. By definition, each dominating set $D$ counted by $Z''_{\overrightarrow{uv}}$ is a dominating set of $\sG(n,p)-u$ of size $r$ such that $v\in D$ and $0\le | N(u)\cap D\setminus\{v\} | \le L-1$.
Therefore, we get
\[
\ex Z''_{\overrightarrow{uv}} = \pr(\Bin(r-1,p)<L) \ \ex_{\sG(n-1,p)} Z_v.
\]
Furthermore, given $u\in V$, by counting in two different ways the number of pairs $D,D'$ of dominating sets of $\sG(n,p)-u$ of size $r$ with one marked vertex $v\in D\cap D'$ such that $0\le | N(u)\cap D\setminus\{v\} | \le L-1$, we get
\[
\sum_{v\in V\setminus\{u\}} \ex ({Z''_{\overrightarrow{uv}}}^2) =  \sum_{i=1}^r i Q_i \ex_{\sG(n-1,p)} W_i, 
\]
where $Q_i$ is defined in~\eqref{eq:Qidef}. Therefore, since the distribution of $Z''_{\overrightarrow{uv}}$ does not depend on $v\in V\setminus\{u\}$,
\[
\ex ({Z''_{\overrightarrow{uv}}}^2) = 
\frac{1}{n-1}\sum_{i=1}^r i Q_i \ex_{\sG(n-1,p)} W_i.
\]
Recall from our assumption on $I'$ that $n-1\in I$. Then, 
applying Proposition~\ref{prop:variancekey} with $n-1$ instead of $n$ (at the expense of an additional, negligible, multiplicative factor $1+O(1/n)$), we get
\begin{align*}
\ex ({Z''_{\overrightarrow{uv}}}^2) =& 
\frac{1}{n-1}\sum_{i=1}^r i Q_i \ex_{\sG(n-1,p)} W_i \\
=&Q_1\frac{r^2}{n^2}\left( 1 + \Theta\left(\frac{\log^4n}{pn} + \frac{\log^3n}{p^2n}\right) \right) (\ex X)^2. 
\end{align*}
Hence, 
\begin{align*}
\var Z''_{\overrightarrow{uv}} =& Q_1\left(\frac{r}{n}\right)^2\left( 1 + \Theta\left(\frac{\log^4n}{pn} + \frac{\log^3n}{p^2n}\right) \right) (\ex X)^2 -Q_1 (\ex_{\sG(n-1,p)} Z_v)^2 \\
=&  Q_1\left(\frac{r}{n}\right)^2\left( 1 + \Theta\left(\frac{\log^4n}{pn} + \frac{\log^3n}{p^2n}\right) \right)(\ex X)^2 -Q_1 \left((1+O(1/n))\frac{r}{n}\ex X \right)^2 \\
=& Q_1 \left(\frac{r}{n}\right)^2 O\left(\frac{\log^4n}{pn} + \frac{\log^3n}{p^2n}\right) (\ex X)^2 \\
=& Q_1 (\ex_{\sG(n-1,p)} Z_v)^2 O\left(\frac{\log^4n}{pn} + \frac{\log^3n}{p^2n}\right)\\
=& \left(\pr(\Bin(r-1,p)<L) \times \ex_{\sG(n-1,p)} Z_v\right)^2 O\left(\frac{\log^4n}{pn} + \frac{\log^3n}{p^2n}\right) \\
=& (\ex Z''_{\overrightarrow{uv}})^2 O\left(\frac{\log^4n}{pn} + \frac{\log^3n}{p^2n}\right),
\end{align*}
and the desired property holds.
\end{proof}

Finally, we proceed to the proof of the main theorem.

\begin{proof}[Proof of Theorem~\ref{thm:main_general}]
Let $\eps>0$ be an arbitrarily small constant such that $n^{-1/3+\eps} \le p \le 1-\eps$, and recall that $G=(V,E)$ denotes a random outcome of $\sG(n,p)$. Corollary~\ref{cor:small_variance} yields immediately the first part of the statement for a less restrictive range of $p$. To prove the second part, we combine the strategy described in Observation~\ref{obs:Y} together with Lemma~\ref{lem:YZ}: our goal is to show that, for some sufficiently small constant $\delta > 0$ (only depending on $\eps$), a.a.s.\ for any set $A$ of at most $\delta np$ edges of $G$, the sum of light and heavy damages of $A$ is strictly less than $\frac34 \ex X$. Thus, since a.a.s.\ $X=(1+o(1)) \ex X$ (by Corollary~\ref{cor:small_variance}), we infer that a.a.s.\ not all dominating sets can be removed by deleting at most $\delta np$ edges of $G$, yielding the desired lower bound on $b(G)$. The upper bound follows from~\eqref{eq:trivialupper1}.

We will first bound the heavy damage of any set of at most $\delta np$ edges in $E$. For convenience, we say that a directed pair $\overrightarrow{uv}$ is present in $G$ (or is a directed edge of $G$) if the corresponding pair $uv$ belongs to $E$. Using Lemma~\ref{lem:var2}, by Chebyshev's inequality and noting that $O\left(\frac{\log^4n}{pn} + \frac{\log^3n}{p^2n}\right)=O\left(\frac{\log^4 n}{p^2 n}\right)$, for any directed pair $\overrightarrow{uv}$ (possibly not present in $G$) and $t>0$
\begin{align}\label{eq:chebyshev0}
\pr \left( \left|Z''_{\overrightarrow{uv}} - \ex Z''_{\overrightarrow{uv}}\right| \geq  t \ex Z''_{\overrightarrow{uv}} \right) \le \frac{\var Z''_{\overrightarrow{uv}}}{\left(t \ex Z''_{\overrightarrow{uv}}\right)^2} = O\left( \frac{\log^4 n}{t^2p^2 n} \right).
\end{align}
Making use of the subsubsequence principle (see~e.g.~\cite{JLR}),
we split the analysis into two cases, depending on the asymptotic behaviour of $p$:
set $\nu=2-\varepsilon/2$ if $p=o(1)$; and $\nu=1+\varepsilon^2$ if $p=\Theta(1)$ and $p \leq 1 -\eps$. Using Lemma~\ref{lem:exp_zuv}, the equation above yields
\begin{align}\label{eq:chebyshev1}
\pr \left( Z''_{\overrightarrow{uv}} \geq \frac{\ex X}{(pn)^{\nu}} \left(1 + t\right) \right) = O\left( \frac{\log^4 n}{t^2p^2 n} \right).
\end{align}
Clearly, \eqref{eq:chebyshev1} implies that, uniformly for all $i\ge1$,
\begin{align}\label{eq:chebyshev2}
\pr \left( Z''_{\overrightarrow{uv}} \geq \frac{\ex X}{(pn)^{\nu}} 2^i \right) = O\left( \frac{\log^4 n}{2^{2i} p^2 n} \right).
\end{align}
We call a directed pair ${\overrightarrow{uv}}$ (possibly not in $G$)  \textbf{$i$-bad} if 
$$
2^i  \frac{\ex X}{(pn)^{\nu}} \leq  Z''_{\overrightarrow{uv}} < 2^{i+1}  \frac{\ex X}{(pn)^{\nu}}
$$
and \textbf{bad} if it is $i$-bad for some $i \geq 1$. Directed pairs that are not bad will be called \textbf{good}.
Observe from its definition that $Z''_{\overrightarrow{uv}}$ is independent of the event that $\overrightarrow{uv}$ is present in $G$. Hence, using~\eqref{eq:chebyshev2}, the probability that a directed pair $\overrightarrow{uv}$ is $i$-bad and is present in $G$ is $p \cdot O\big( \frac{\log^4 n}{2^{2i} p^2 n} \big)$, and therefore
the heavy damage of all bad directed edges in the graph $G$ is in expectation at most
\begin{align*}
 n^2 p \sum_{i \geq 1 } 2^{i+1} \frac{\ex X}{(pn)^{\nu}} \; O \left( \frac{\log^4 n}{2^{2i} p^2 n} \right) &= O \left( \frac{\ex X \log^4 n}{p^{1+\nu} n^{\nu-1}} \right) \sum_{i \geq 0} 2^{-i}
 = O \left (\frac {\ex X \log^4 n}{p^{1+\nu} n^{\nu-1}} \right). 
\end{align*}
Consequently, by Markov's inequality, the heavy damage of all bad directed edges in $G$ is a.a.s.\  at most $\ex X \log^5 n/(p^{1+\nu} n^{\nu-1}) = o(\ex X)$, as long as $p \geq n^{-1/3+\varepsilon}$ and $p \leq 1 - \varepsilon$. On the other hand, the heavy damage of a good directed pair is at most $2\frac{\ex X}{(pn)^{\nu}}$ by definition. Therefore, given any set $A$ of pairs (possibly not in $G$) of size at most $\delta pn$, the heavy damage of the set of all good directed pairs $\overrightarrow{uv}$ such that $uv\in A$ is deterministically at most
\[
2\delta pn \cdot 2\frac{\ex X}{(pn)^{\nu}} = O \Big(   \frac{\ex X}{(pn)^{\nu-1}} \Big) = o(\ex X),
\]
where in the last step we used again our assumptions on $p$. 
Putting all the above together, we conclude that a.a.s.\ the heavy damage of any set of edges $A$ in the graph $G$ with $|A| \le \delta np$ is
\begin{equation}\label{eq:heavy}
Z^{''}_A = o(\ex X).
\end{equation}

Now we proceed to bound the light damage of any set $A$ of edges in $G$ of size at most $\delta np$.
The analysis bears some similarities to our previous estimation of the heavy damage, but the role of directed pairs will be taken by vertices.

Using Lemma~\ref{lem:vertex}, by Chebyshev's inequality and noting that $\var Z_v = O\left(\frac{\log^4n}{p^2n}\right) (\ex Z_v)^2$, for any vertex $v$ and $s>0$,
\[
\pr \left( |Z_v - \ex Z_v| \geq s \ex Z_v \right)  \le \frac{\var Z_v}{\left(s\ex Z_v\right)^2}= O\left(\frac{\log^4 n}{s^2 p^2 n}\right).
\]
Thus,
\begin{equation}\label{eq:chebyshev1b}
\pr \left( Z_v \geq \ex X \frac{r}{n} (1+s) \right) = O\left(\frac{\log^4 n}{s^2 p^2 n}\right).
\end{equation}
Clearly,~\eqref{eq:chebyshev1b} implies that, uniformly for all $i\ge1$,
\begin{equation}\label{eq:chebyshev2b}
\pr \left( Z_v \geq \ex X \frac{r}{n} 2^i \right) = O\left(\frac{\log^4 n}{2^{2i} p^2 n}\right).
\end{equation}
We call a vertex $v$ \textbf{$i$-exceptional}, if
$$
 2^i \ex X \frac{r}{n} \leq Z_v \leq 2^{i+1} \ex X \frac{r}{n},
$$
and \textbf{exceptional} if it is $i$-exceptional for some $i \geq 1$.
Let $V_1\subseteq V$ be the set of all exceptional vertices in the graph $G$. We want to bound the number $\sum_{v\in V_1} Z_v$ of dominating sets of size $r$ containing at least one exceptional vertex. Since we are summing over a random set, it is convenient to interpret the previous sum as
\[
\sum_{v\in V_1} Z_v = \sum_{v\in V} Z_v 1_{\{v\in V_1\}} ,
\]
where $1_{\{v\in V_1\}}$ is the indicator function of the event that $v\in V_1$.
Hence, in view of~\eqref{eq:chebyshev2b} and by the linearity of expectation, 
\[
\ex \bigg( \sum_{v\in V_1} Z_v \bigg)
 = n \sum_{i\geq1} 2^{i+1} \ex X \frac{r}{n} \; O \left( \frac{\log^4 n}{2^{2i}p^2 n}\right) =
O \left( \ex X \frac{r \log^4 n}{p^2 n}\right) \sum_{i\geq 0} 2^{-i}
=O \left(\ex X \frac{\log^5 n}{p^3 n}\right).
\]
Thus, by Markov's inequality, a.a.s.\ 
\begin{equation}\label{eq:ZV1}
\sum_{v\in V_1} Z_v \le \ex X \frac{\log^6 n}{p^3 n}.
\end{equation}
We call a vertex \textbf{normal} if it is not exceptional. For a normal vertex $v$, $Z_v$ is at most  $2\ex X \frac{r}{n}$ by definition. 

Now, for a given set $A$ of edges in $G$ of size at most $\delta pn$, let $V(A)$ be the set of vertices containing all endpoints of edges from $A$, that is, 
\[
V(A)=\{ v \in V \mid \exists e \in A \mbox { such that } v \in e\}.
\]
Partition $V(A)$ as follows: $V(A)=V_0(A) \cup V_1(A)$, where $V_0(A)$ is the subset of normal vertices, and $V_1(A)$ the subset of exceptional vertices.
From the definition of light damage, we get 
\begin{equation}\label{eq:Alight}
Z'_A \le \sum_{v \in V(A)} Z_v/(\varepsilon^2 p r)
= \sum_{v \in V_0(A)} Z_v/(\varepsilon^2 p r) + \sum_{v \in V_1(A)} Z_v/(\varepsilon^2 p r).
\end{equation}
For the first sum on the RHS of~\eqref{eq:Alight}, we have
$$
\sum_{v \in V_0(A)} \frac {Z_v}{\varepsilon^2 p r} \leq \frac{2|V_0(A)|\ex X}{\varepsilon^2 p n} \leq \frac{4|A|\ex X}{\varepsilon^2 p n} \leq \frac{4\delta\ex X}{\varepsilon^2} \le \ex X/2,
$$
deterministically and regardless of the choice of $A$, as long as $\delta\le\eps^2/8$. On the other hand, for the second sum on the RHS of~\eqref{eq:Alight}, we use~\eqref{eq:ZV1}, and obtain that a.a.s.\ for every choice of $A$
\[
 \sum_{v \in V_1(A)} \frac {Z_v}{\varepsilon^2 p r}
\le  \sum_{v \in V_1} \frac {Z_v}{\varepsilon^2 p r}
\leq \frac{1}{\varepsilon^2 p r} \ex X \frac{ \log^6 n}{p^3 n}
= O \left( \ex X \frac{ \log^5 n}{p^3 n} \right) = o(\ex X),
\]
since $p\geq n^{-1/3+\varepsilon}$. Hence, a.a.s.\ for every set $A$ of edges of $G$ with $|A| \leq \delta np$, we have 
$$
Z'_A \leq (1+o(1)) \frac {\ex X}{2}.
$$
Combining this, \eqref{eq:heavy} and Corollary~\ref{cor:small_variance}, we conclude that a.a.s.\ for every choice of $A$,
\[
Z_A < \frac34\ex X < X,
\]
as required. The second part of the statement follows and the proof is finished.
\end{proof}

\bibliographystyle{alpha}

\end{document}